\documentclass[12pt]{amsart}

\usepackage{amssymb}

\newtheorem{theorem}{Theorem}
\newtheorem{remark}{Remark}
\newtheorem{corollary}{Corollary}

\def\const{\mathop{\rm const.}}

\def\sign{\mathop{\rm Sign}}

\begin{document}

\baselineskip=17pt

\title{The equi-affine and Frenet curvatures of curves in pseudo-Riemannian 2-manifolds}

\author{Karina Olszak and Zbigniew Olszak}

\dedicatory{Dedicated to the memory of Professor Witold Roter}

\address{Department of Mathematics (Karina Olszak) and Department of \linebreak Applied Mathematics (Zbigniew Olszak), Wroc{\l}aw University of Science and Technology, Wybrze\.ze Wyspia\'nskiego 27, 50-370 Wroc{\l}aw, Poland}

\date{}

\begin{abstract}
For an arbitrary nondegenerate curve in a pseudo-Riemann\-ian (including Riemannian) 2-manifold, we express the equi-affine curvature with the help of the Frenet (geodesic) curvature of this curve. 
\end{abstract}



\maketitle

\section{Introduction}

The Frenet (geodesic) and the equi-affine curvatures of curves (mainly, in the Euclidean or affine plane) play a fundamental role and have important applications in analytic, geometric, technical and biological investigations. Recently, among others, investigations concerning affine plane curves, affinely invariant numerical approximations, geometric approximation of curves and the study of the arm movements of humans and other primates appeared in \cite{COT,Dym,FH,Gh,ST}, etc. 

It is a natural question to find relations between the Frenet (geodesic) curvature and the equi-affine curvature of a nondegenerate curve in an arbitrary pseudo-Riemannian (including Riemannian) 2-manifold. In the present paper, our goal is to establish a precise relation between those curvatures. Our result generalizes certain partial results known in the case of curves in the Euclidean plane (cf.\ \cite{FH,Gh,ST}). From our formula, it can be seen that the constant Frenet curvature implies the constant equi-affine curvature. We also provide examples of curves with non-constant Frenet curvature and constant equi-affine curvature in certain Riemannian and Lorentzian 2-manifolds. But first, following the classical ideas, we arrive at formulas for the equi-affine curvature in an arbitrary equi-affine 2-manifold, and for the Frenet curvature in an arbitrary pseudo-Riemannian 2-manifold. 

\section{Curves in equi-affine 2-manifolds}

In this section, following the classical books (see \cite{BR,Bu,Gu,NS}) and certain recent papers (see \cite{FT,IS}, etc.), we present the notion of the equi-affine curvature of a curve in an arbitrary 2-dimensional equiaffine manifold.

Let $(M^2,\nabla,\varOmega)$ be an equi-affine 2-manifold, that is, $M^2$ is a connected, 2-dimensional differentiable manifold endowed with a symmetric affine connection $\nabla$ and a volume form $\varOmega$, which is parallel with respect to $\nabla$ (for this notion, cf.\ e.g.\ \cite{NS}).

By a nondegenerate curve in $(M^2,\nabla,\varOmega)$, we mean a (smooth) mapping $\alpha\colon I\to M^2$, $I$ being an open interval, such that $\varOmega(\alpha^{\,\prime}(t), (\nabla_{\alpha^{\,\prime}}\alpha^{\,\prime})(t)) \neq 0$ for any $t\in I$. $\sign(\varOmega(\alpha^{\,\prime}, \nabla_{\alpha^{\,\prime}}\alpha^{\,\prime}))\;(=\pm1)$ is the orientation of the frame $(\alpha^{\,\prime},\nabla_{\alpha^{\,\prime}}\alpha^{\,\prime})$ with respect to the volume form $\varOmega$. 

Let $\alpha\colon I\to M^2$ be a nondegenerate curve in an equi-affine 2-manifold $(M^2,\nabla,\varOmega)$. Let $h\colon I\to J$, $J$ being an open interval, be a smooth function onto $J$ such that $h^{\,\prime}\neq0$ for any $t\in I$. Consider the reparametrization $\overline\alpha\colon J\to M^2$ of the curve $\alpha$ defined by $\alpha=\overline{\alpha}\circ h$. Denote the new parameter by $\sigma$ and write $\sigma(t)$ instead of $h(t)$. Then,  we have $\alpha^{\,\prime}(t) = \overline{\alpha}^{\,\prime}(\sigma(t)) \sigma^{\,\prime}(t)$. Removing the argument $t$, we can write 
\begin{align*}
  \varOmega(\alpha^{\,\prime},
            \nabla_{\alpha^{\,\prime}}\alpha^{\,\prime})
  = \sigma^{\,\prime\,3} 
    (\varOmega(\overline{\alpha}^{\,\prime},
              \nabla_{\overline{\alpha}^{\,\prime}}
                   \overline{\alpha}^{\,\prime}))\circ\sigma. 
\end{align*}
Let us call the new parameter $\sigma$ the equi-affine arclength parameter if $\varOmega(\overline{\alpha}^{\,\prime}, \nabla_{\overline{\alpha}^{\,\prime}} \overline{\alpha}^{\,\prime}) = 1$. It is obvious that this condition holds if and only if 
\begin{align*}
  \sigma^{\,\prime} = \sqrt[3]{\varOmega(\alpha^{\,\prime},
            \nabla_{\alpha^{\,\prime}}\alpha^{\,\prime})}.
\end{align*}
The equi-affine arclength parameter $\sigma$ starting at $t_0\in I$ is
\begin{align*}
  \sigma(t) = \int^t_{t_0}
         \sqrt[3]{\varOmega(\alpha^{\,\prime}(u),
            (\nabla_{\alpha^{\,\prime}}\alpha^{\,\prime})(u))}\,du, \ t\in I.
\end{align*}

In the sequel, $\overline{\alpha}\colon J\to M^2$ denotes an equi-affine arclength parameterizaton of $\alpha$. Since $\varOmega(\overline{\alpha}^{\,\prime}, \nabla_{\overline{\alpha}^{\,\prime}} \overline{\alpha}^{\,\prime}) = 1$, one has $\varOmega(\overline{\alpha}^{\,\prime}, \nabla^2_{\overline{\alpha}^{\,\prime}} \overline{\alpha}^{\,\prime}) = 0$, where $\nabla^2_{\overline{\alpha}^{\,\prime}} \overline{\alpha}^{\,\prime} = \nabla_{\overline{\alpha}^{\,\prime}} (\nabla_{\overline{\alpha}^{\,\prime}}\overline{\alpha}^{\,\prime})$. Therefore, $\nabla^2_{\overline{\alpha}^{\,\prime}} \overline{\alpha}^{\,\prime}$ and  $\overline{\alpha}^{\,\prime}$ are linearly dependent along the curve. The function $\overline{\kappa}_a$ on $J$ such that 
\begin{align}
\label{frenov}
  \nabla^2_{\overline{\alpha}^{\,\prime}} \overline{\alpha}^{\,\prime}
  = - \overline{\kappa}_a \overline{\alpha}^{\,\prime}
\end{align}
is said to be the equi-affine curvature of the curve $\overline{\alpha}$. $\overline{\kappa}_a$ can be calculated in the following way:
\begin{align*}
  \overline{\kappa}_a 
  = \varOmega(\nabla_{\overline{\alpha}^{\,\prime}} \overline{\alpha}^{\,\prime},
              \nabla^2_{\overline{\alpha}^{\,\prime}} \overline{\alpha}^{\,\prime}).
\end{align*}

Now, returning to the arbitrary parametrization $\alpha\colon I\to M^2$, consider two auxiliary functions $\mu$ and $\psi$ defined on $I$ by 
\begin{align*}
  \mu = \sigma^{\,\prime} = \sqrt[3]{\varOmega(\alpha^{\,\prime},
            \nabla_{\alpha^{\,\prime}}\alpha^{\,\prime})}, \quad
  \psi = \mu^{-1}.
\end{align*}
Define the Cartan moving frame $({\mathbf e}_1,{\mathbf e}_2)$ and the equi-affine curvature $\kappa_a$ of the curve $\alpha$ by
\begin{align*}
  {\mathbf e}_1 = \overline{\alpha}^{\,\prime} \circ \sigma,\quad
  {\mathbf e}_2 = (\nabla_{\overline{\alpha}^{\,\prime}}
                     \overline{\alpha}^{\,\prime}) \circ \sigma,\quad
  \kappa_a = \overline{\kappa}_a \circ \sigma.
\end{align*}
By direct calculations, we find that 
\begin{align*}
  \nabla_{\alpha^{\,\prime}}{\mathbf e}_1 
  &= \nabla_{\alpha^{\,\prime}}(\overline{\alpha}^{\,\prime} \circ \sigma)
   = \sigma^{\,\prime}
     \big((\nabla_{\overline{\alpha}^{\,\prime}}\overline{\alpha}^{\,\prime}) \circ \sigma\big)
   = \mu {\mathbf e}_2, \\
  \nabla_{\alpha^{\,\prime}}{\mathbf e}_2 
  &= \nabla_{\alpha^{\,\prime}}
     \big((\nabla_{\overline{\alpha}^{\,\prime}}\overline{\alpha}^{\,\prime}) \circ \sigma\big)
   = \sigma^{\,\prime} 
     \big((\nabla^2_{\overline{\alpha}^{\,\prime}}\overline{\alpha}^{\,\prime}) \circ \sigma\big)\\
  &= \sigma^{\,\prime} 
     \big((- \overline{\kappa}_a \overline{\alpha}^{\,\prime}) \circ \sigma\big)
   = - \sigma^{\,\prime} 
     (\overline{\kappa}_a \circ \sigma) (\overline{\alpha}^{\,\prime} \circ \sigma)
   = - \mu \kappa_a {\mathbf e}_1,
\end{align*}
where we have also applied (\ref{frenov}). 

Thus, the equi-affine Frenet type formulas of the curve $\alpha$ are 
\begin{equation}
\label{afffren}
  \nabla_{\alpha^{\,\prime}}{\mathbf e}_1 = \mu {\mathbf e}_2,\quad
		\nabla_{\alpha^{\,\prime}}{\mathbf e}_2 = - \mu \kappa_a {\mathbf e}_1.
\end{equation}
Since $\varOmega(\mathbf e_1, \mathbf e_2)=1$, using the second formula of (\ref{afffren}), we conclude 
\begin{align}
\label{arbkap}
  \kappa_a 
  = \psi \varOmega({\mathbf e}_2, \nabla_{\alpha^{\,\prime}}{\mathbf e}_2).
\end{align}

Now, we are going to find an explicit formula for the curvature $\kappa_a$. 

First, having (\ref{afffren}), we find 
\begin{align}
\label{e12}
  {\mathbf e}_1 
    = \psi \alpha^{\,\prime}, \quad
  {\mathbf e}_2 
    = \psi \nabla_{\alpha^{\,\prime}}{\mathbf e}_1 
    = \psi \psi^{\,\prime} \alpha^{\,\prime} + 
          \psi^2 \nabla_{\alpha^{\,\prime}}\alpha^{\,\prime}.
\end{align}
Next, by covariant differentiation along the curve, we get 
\begin{equation}
\label{der2}
  \nabla_{\alpha^{\,\prime}}{\mathbf e}_2
  = \big(\psi^{\,\prime\,2} + \psi \psi^{\,\prime\prime}\big) \alpha^{\,\prime}
    + 3 \psi \psi^{\prime}\, \nabla_{\alpha^{\,\prime}}\alpha^{\,\prime}
    + \psi^2 \nabla^2_{\alpha^{\,\prime}}\alpha^{\,\prime}.
\end{equation}

On the other hand, applying (\ref{e12}) and (\ref{der2}) to (\ref{arbkap}), we obtain
\begin{align}
\label{cexpl}
  \kappa_a 
    = \ & (2 \psi^3 \psi^{\,\prime\,2} - \psi^4 \psi^{\,\prime\prime}) 
      \varOmega(\alpha^{\,\prime},\nabla_{\alpha^{\,\prime}}{\alpha^{\,\prime}}) \\
      & \null+ \psi^4 \psi^{\,\prime}
         \varOmega(\alpha^{\,\prime},\nabla^2_{\alpha^{\,\prime}}{\alpha^{\,\prime}})
       + \psi^5 \varOmega(\nabla_{\alpha^{\,\prime}}\alpha^{\,\prime},
                            \nabla^2_{\alpha^{\,\prime}}{\alpha^{\,\prime}}). \nonumber
\end{align}
But $\psi^3 \varOmega (\alpha^{\,\prime},\nabla_{\alpha^{\,\prime}}{\alpha^{\,\prime}}) = 1$, and consequently, $\psi^4 \varOmega (\alpha^{\,\prime},\nabla^2_{\alpha^{\,\prime}}{\alpha^{\,\prime}}) = - 3 \psi^{\,\prime}$. The last two equalities turn formula (\ref{cexpl}) into 
\begin{align*}
  \kappa_a = \null - \psi^{\,\prime\,2} - \psi \psi^{\,\prime\prime} 
       + \psi^5 \varOmega
          (\nabla_{\alpha^{\,\prime}}\alpha^{\,\prime},
           \nabla^2_{\alpha^{\,\prime}}{\alpha^{\,\prime}}). 
\end{align*}

Thus, we can state the following theorem.

\begin{theorem}
The equi-affine curvature $\kappa_a$ of a curve in an equi-affine 2-manifold is given by the formula 
\begin{align}
\label{curvexpl}
  \kappa_a = \null-\frac12 \big(\psi^2\big)^{\prime\prime}
       + \psi^5 \varOmega
          (\nabla_{\alpha^{\,\prime}}\alpha^{\,\prime},
           \nabla^2_{\alpha^{\,\prime}}{\alpha^{\,\prime}}). 
\end{align}
\end{theorem}

The second equi-affine Frenet type equation (\ref{afffren}), with the aid of (\ref{e12}) and (\ref{der2}), leads to the following theorem. 

\begin{theorem}
The equi-affine curvature $\kappa_a$ of a curve in an equi-affine 2-manifold satisfies the differential equation
\begin{align*}
   \psi^2 \nabla^2_{\alpha^{\,\prime}}\alpha^{\,\prime} 
   + 3 \psi \psi^{\prime}\, \nabla_{\alpha^{\,\prime}}\alpha^{\,\prime}
   + \big(\psi^{\,\prime\,2} + \psi \psi^{\,\prime\prime}
          + \kappa_a \big) \alpha^{\,\prime} = 0.
\end{align*}
\end{theorem}

As a consequence of Theorem 2, we obtain the following 

\begin{corollary}
If $\psi=\const$ (equivalently, $\mu=\const$), then the equi-affine curvature $\kappa_a$ of a curve in an equi-affine 2-manifold satisfies the differential equation
\begin{align*}
  \psi^2 \nabla^2_{\alpha^{\,\prime}}\alpha^{\,\prime} + \kappa_a \alpha^{\,\prime}=0.
\end{align*}
\end{corollary}

\section{Pseudo-Riemannian 2-manifolds}

Let $(M^2,g)$ be a pseudo-Riemannian $2$-manifold. For our purposes, we always assume that $(M^2,g)$ is connected and oriented. Assign a parameter $\omega(=\pm1)$ to the metric $g$ by setting $\omega=1$ if the signature of the metric $g$ is $(++)$ or $(--)$, and $\omega=-1$ if the signature of the metric $g$ is $(+-)$.

Denote by $\varOmega$ the natural volume 2-form generated by the metric $g$, and assume that $\varOmega$ is compatible with the orientation of the manifold $M^2$ (see e.g. \cite{ON}). The form $\varOmega$ is given locally by
\begin{equation*}
  \varOmega = 2 \sqrt{|G|}\, dx^1 \wedge dx^2
            = \sqrt{|G|}\, (dx^1 \otimes dx^2 - dx^2 \otimes dx^1),
\end{equation*}
where $(x^1,x^2)$ are local coordinates in a chart belonging to the oriented atlas of the manifold, and $g_{ij}=g(\partial/\partial x^i,\partial/\partial x^j)$ and $G = g_{11}g_{22} - g_{12}^2$. Note that $|G|=\omega G$. 

For a pair of tangent vectors $u,v\in T_pM^2$, $p\in M^2$, one has: (a) $\varOmega(u,v)\neq 0$ if and only if $u,v$ are linearly independent; (b) $\varOmega(u,v)>0$ or $\varOmega(u,v)<0$ if and only if the orientation of the frame $(u,v)$ is positive or negative, respectively; (c) $\varOmega(u,v)=\pm1$ when $u,v$ are orthonormal. 

On the manifold $M^2$, define a $(1,1)$-tensor field $J$ by
\begin{align}
\label{defJ}
  g(X,JY) = \varOmega(X,Y),\ X,Y\in\mathfrak X(M),
\end{align}
$\mathfrak X(M)$ being the Lie algebra of vector fields on $M^2$. In local coordinates $(x^1,x^2)$, the defining condition takes the form $g_{is} {J^s}_j = \varOmega_{ij}$ (Einstein's summation convention is used). Therefore, ${J^i}_j = g^{is}\varOmega_{sj}$, and consequently
\begin{align*}
  {J^1}_1 = - {J^2}_2 = \omega \frac{g_{12}}{\sqrt{|G|}},\quad 
		{J^2}_1 = - \omega \frac{g_{11}}{\sqrt{|G|}},\quad 
		{J^1}_2 = \omega \frac{g_{22}}{\sqrt{|G|}}. 
\end{align*}
It is now straightforward to verify that ${J^i}_s {J^s}_j = - \omega \mathop{\it{Id}}^i_j$, that is, 
\begin{align}
\label{jj}
  J^2 = -\omega\mathop{\it{Id}}. 
\end{align}
This means that $J$ is an almost complex structure in the case $\omega=1$, and it is an almost paracomplex structure in the case $\omega=-1$. The relations 
\begin{align}
\label{gjjg}
  \varOmega(X,JY) = - g(JX,JY) = - \omega g(X,Y), \ X,Y\in\mathfrak X(M), 
\end{align}
are consequences of (\ref{defJ}) and (\ref{jj}). Thus, the pair $(J,g)$ is an almost Hermitian structure in the case where $\omega=1$, and an almost para-Hermitian structure in the case where $\omega=-1$.

In fact, the pair $(J,g)$ becomes a K\"ahlerian or para-K\"ahlerian structure on $M^2$ according to $\omega=+1$ or $\omega=-1$, since in the both cases, $J$ is integrable and $\varOmega$ is a symplectic form on $M^2$. Thus, $\nabla J=0$, where $\nabla$ is the Levi-Civita connection of $(M^2,g)$. For various applications of para-K\"ahlerian structures, see \cite{CFG,IT,JR}, etc.


\section{Frenet curves in pseudo-Riemannian 2-manifolds}

Similarly as in the Euclidean plane and in the Minkowski plane (cf. \cite{BB,GAS,SO}, etc.), we will use the almost (para)complex structure $J$ to define the normal vector field and to describe the curvature of a curve on a pseudo-Riemannian 2-manifold. For the classical method, see e.g.\ \cite{H}.

Let $(M^2,g)$ be a pseudo-Riemannian 2-manifold. Let $\alpha$ be a curve in $(M^2,g)$, that is, a smooth mapping $\alpha\colon I\to M^2$, $I$ being an open interval. We will consider only non-singular curves, that is, curves for which $g(\alpha^{\,\prime}(t),\alpha^{\,\prime}(t))\neq0$ for any $t\in I$. The speed function $\nu\colon I\to\mathbb R^{+}$ is defined by 
\begin{equation}
\label{nu}
  \nu(t) = |g(\alpha^{\,\prime}(t),\alpha^{\,\prime}(t))|^{1/2},\quad t\in I,
\end{equation}
and the arclength parameter $s$ starting at $t_0\in I$ is 
\begin{align*}
  s(t) = \int_{t_0}^t \nu(u)\,du, \ t\in I.
\end{align*}

The curve $\alpha$ is called a geodesic in the manifold $(M^2,g)$ (possibly, under a suitable changing of the parametrization) if and only if the vector fields $\nabla_{\alpha^{\,\prime}}\alpha^{\,\prime}$ and $\alpha^{\,\prime}$ are linearly dependent at each point of the curve. As is well-known, $\alpha$ is a geodesic if and only if $\nabla_{\alpha^{\,\prime}}\alpha^{\,\prime} = (\nu^{\,\prime}/\nu) \alpha^{\,\prime}$, or equivalently, $\varOmega(\alpha^{\,\prime},\nabla_{\alpha^{\,\prime}}\alpha^{\,\prime})=0$. 

Along the curve $\alpha$, the unit tangent vector field is defined by
\begin{align}
\label{e1-riem}
  \mathbf T = \dfrac{\alpha^{\,\prime}}{\nu}
		     = \dfrac{\alpha^{\,\prime}}{|g(\alpha^{\,\prime},\alpha^{\,\prime})|^{1/2}}.
\end{align}
As the normal vector field along the curve $\alpha$, we choose 
\begin{align}
\label{e2-riem}
  \mathbf N = - \omega \varepsilon J\mathbf T, 
\end{align}
where $\varepsilon=g(\mathbf T,\mathbf T)(=\pm1)$ and $J$ is the almost (para)complex structure defined in the previous section. Using (\ref{defJ}) and (\ref{gjjg}), we check that $g(\mathbf N,\mathbf N) = \omega \varepsilon$, $g(\mathbf T,\mathbf N) = 0$, $\varOmega(\mathbf T,\mathbf N)=1$. Thus, $\mathbf N$ is the unique unit vector field along the curve which is orthogonal to $\mathbf T$ and such that the pair $(\mathbf T,\mathbf N)$ is positively oriented. 

Since the vector field $\nabla_{\mathbf T}\mathbf T$ is orthogonal to $\mathbf T$ at every point of the curve, there exists a function $\kappa_r$ on $I$ such that 
\begin{align}
\label{fren1}
  \nabla_{\mathbf T}\mathbf T = \kappa_r \mathbf N.
\end{align}
The function $\kappa_r$ is called the (oriented, Frenet) curvature of the curve $\alpha$. The pair $(\mathbf T,\mathbf N)$ is the Frenet frame and (\ref{fren1}) is the Frenet equation of the curve $\alpha$. 

By direct calculations, in which we use among others, (\ref{e1-riem}), (\ref{e2-riem}) and (\ref{fren1}), we find 
\begin{align*}
  \kappa_r = \varOmega({\mathbf T}, \nabla_{\mathbf T}{\mathbf T}) 
           = \frac{1}{\nu^2} 
              \varOmega\Big(\alpha^{\,\prime}, 
                  \nabla_{\alpha^{\,\prime}}\Big(\frac{\alpha^{\,\prime}}{\nu}\Big)\Big)
           = \frac{1}{\nu^3} 
              \varOmega(\alpha^{\,\prime}, 
                  \nabla_{\alpha^{\,\prime}}\alpha^{\,\prime}).
\end{align*}

Thus, we have obtained the following formula for the curvature of a Frenet curve. 

\begin{theorem}
The curvature of a Frenet curve in a pseudo-Riemannian 2-manifold $(M^2,g)$ is given by the formula
\begin{align}
\label{kap5}
  \kappa_r = \frac{\varOmega(\alpha^{\,\prime},\nabla_{\alpha^{\,\prime}}\alpha^{\,\prime})}
                {|g(\alpha^{\,\prime},\alpha^{\,\prime})|^{3/2}}
         = - \frac{g(\nabla_{\alpha^{\,\prime}}\alpha^{\,\prime},J\alpha^{\,\prime})}
                  {|g(\alpha^{\,\prime},\alpha^{\,\prime})|^{3/2}},
\end{align}
where $\varOmega$ is the natural volume element on $(M^2,g)$ and $J$ is the natural almost (para)complex structure generated by $\varOmega$. 
\end{theorem}

\section{The equi-affine and Frenet curvatures}

In this section, we determine a relation between the Frenet curvature and equi-affine curvature of a curve in a pseudo-Riemannian 2-manifold.

Let $(M^2,g)$ be a pseudo-Riemannian 2-manifold. The metric $g$ generates on $M^2$ the equi-affine structure $(\nabla,\varOmega)$ with $\nabla$ being the Levi-Civita connection and $\varOmega$ the natural volume element with respect to the metric $g$. 

Thus, for a nondegenerate curve in $(M^2,g)$, we have (i) the equi-affine curvature $\kappa_a$, which according to (\ref{curvexpl}), we rewrite as 
\begin{align}
\label{kap-aff}
  \kappa_a &= \null-\frac12 \big(\psi^2\big)^{\prime\prime}
       + \psi^5 \varOmega
          (\nabla_{\alpha^{\,\prime}}\alpha^{\,\prime},
           \nabla^2_{\alpha^{\,\prime}}{\alpha^{\,\prime}}),
\end{align}
where $\psi$ is the auxiliary function defined by 
\begin{align}
\label{phi-aff}
   \psi = \big(\varOmega(\alpha^{\,\prime},
            \nabla_{\alpha^{\,\prime}}\alpha^{\,\prime})\big)^{-1/3};
\end{align}
and (ii) the Frenet curvature $\kappa_r$, which according to (\ref{kap5}), we rewrite as 
\begin{align}
\label{kap-riem}
  \kappa_r = \frac{\varOmega(\alpha^{\,\prime},
                   \nabla_{\alpha^{\,\prime}}\alpha^{\,\prime})}
                {|g(\alpha^{\,\prime},\alpha^{\,\prime})|^{3/2}}.
\end{align}

The following theorem states a relation between those curvatures. 

\begin{theorem}
The equi-affine curvature $\kappa_a$ and the Frenet curvature $\kappa_r$ of a regular curve in a pseudo-Riemannian 2-manifold $(M^2,g)$ are related by 
\begin{align}
\label{equfr}
  \kappa_a = \frac{1}{9 {\kappa_r}^{8/3}}
             \big(3 \nu^{-2} \kappa_r {\kappa_r}^{\prime\prime} 
                  - 5 \nu^{-2} {\kappa_r}^{\prime\,2}
                  - 3 \nu^{-3} \nu^{\,\prime} \kappa_r {\kappa_r}^{\prime} 
                  + 9 \omega {\kappa_r}^4\big). 
\end{align}
\end{theorem}

\begin{proof}
First, by (\ref{kap-riem}) and (\ref{nu}), we have 
\begin{align}
\label{om-ar}
  \varOmega(\alpha^{\,\prime},\nabla_{\alpha^{\,\prime}}\alpha^{\,\prime}) = \nu^3 \kappa_r, 
\end{align}
which, applied to (\ref{phi-aff}), gives 
\begin{align}
\label{phi-ar}
  \psi = \nu^{-1} {\kappa_r}^{-1/3}, 
\end{align}
and hence we find that 
\begin{align}
\label{phi2-ar}
  (\psi^2)^{\prime\prime} 
  &= (6 \nu^{-4} \nu^{\,\prime\,2} - 2 \nu^{-3} \nu^{\,\prime\prime}) {\kappa_r}^{-2/3}
     + \frac83 \nu^{-3} \nu^{\,\prime} {\kappa_r}^{-5/3} {\kappa_r}^{\prime} \\
  &\quad\ + \frac{10}{9} \nu^{-2} {\kappa_r}^{-8/3} {\kappa_r}^{\prime\,2} 
     - \frac23 \nu^{-2} {\kappa_r}^{-5/3} {\kappa_r}^{\prime\prime}, \nonumber
\end{align}
Moreover, from (\ref{e1-riem}), (\ref{e2-riem}) and (\ref{fren1}), we get
\begin{align}
\label{aequ-1}
  \nabla_{\alpha^{\,\prime}}\alpha^{\,\prime} 
   = \nu^{-1} \nu^{\,\prime} \alpha^{\,\prime} 
     - \omega \varepsilon \nu \kappa_r J \alpha^{\,\prime},
\end{align}
which, by the covariant differentiation and applying (\ref{jj}), gives 
\begin{align}
\label{aequ-2}
  \nabla^2_{\alpha^{\,\prime}}\alpha^{\,\prime} 
    = (\nu^{-1} \nu^{\,\prime\prime} - \omega \nu^2 {\kappa_r}^2) \alpha^{\,\prime} 
       - \omega \varepsilon (3\nu^{\,\prime} \kappa_r + \nu {\kappa_r}^{\prime}) J\alpha^{\,\prime}. 
\end{align}
We also need the relation  
\begin{align}
\label{aequ-3}
   \varOmega(\alpha^{\,\prime},J{\alpha^{\,\prime}}) 
   = - \omega g(\alpha^{\,\prime},\alpha^{\,\prime}) 
   = - \omega \varepsilon \nu^2, 
\end{align}
which is a consequence of (\ref{gjjg}) and (\ref{nu}). Now, with the aid of (\ref{aequ-1}) - (\ref{aequ-3}) and (\ref{om-ar}), we obtain 
\begin{align}
\label{om2-ar}
  &\varOmega(\nabla_{\alpha^{\,\prime}}\alpha^{\,\prime},
           \nabla^2_{\alpha^{\,\prime}}{\alpha^{\,\prime}}) \\
  &\quad= \omega \varepsilon 
          \big((\nu^{\,\prime\prime} - 3 \nu^{-1} \nu^{\,\prime\,2}) \kappa_r 
               - \nu^{\,\prime} {\kappa_r}^{\prime} 
               - \omega \nu^3 {\kappa_r}^3\big) 
          \varOmega(\alpha^{\,\prime},J{\alpha^{\,\prime}}) \nonumber \\
  &\quad= (3 \nu \nu^{\,\prime\,2} - \nu^2 \nu^{\,\prime\prime}) \kappa_r 
           + \nu^2 \nu^{\,\prime} {\kappa_r}^{\prime} 
           + \omega \nu^5 {\kappa_r}^3. \nonumber
\end{align}
Finally, applying (\ref{phi-ar}), (\ref{phi2-ar}) and (\ref{om2-ar}) to (\ref{kap-aff}), we obtain formula (\ref{equfr}).
\end{proof}

When the curve $\alpha$ is parametrized by the arclength ($\nu=1$), formula (\ref{equfr}) looks much simpler. Namely, we have the following theorem.

\begin{theorem}
The equi-affine curvature $\kappa_a$ and the Frenet curvature $\kappa_r$ of a regular, arclength parametrized curve in a pseudo-Riemannian 2-manifold $(M^2,g)$ are related by the equations
\begin{align}
\label{equfra}
  \kappa_a &= \frac{1}{9 {\kappa_r}^{8/3}} 
             \big(3 \kappa_r {\kappa_r}^{\prime\prime} 
                  - 5 {\kappa_r}^{\prime\,2} + 9 \omega {\kappa_r}^4\big) \\
           &= \null- \frac12 \Big(\frac{1}{{\kappa_r}^{2/3}}\Big)^{\prime\prime}
              + \omega {\kappa_r}^{4/3}. \nonumber
\end{align}
\end{theorem}

\begin{remark}
In the case of curves in the Euclidean plane, formula (\ref{equfra}) has appeared in \cite[Equ.\ (28)]{FH}, \cite[Remark 2.2.4]{Gh}, \cite[Remark 6]{ST}. 
\end{remark}

For a curve on a pseudo-Riemannian 2-manifold, it follows from formula (\ref{equfra}) that if $\kappa_r$ is constant, then $\kappa_a = \omega {\kappa_r}^{4/3}$ is constant too. However, the converse is not true. To see this, it suffices to recall that in the Euclidean plane as well as in the Minkowski plane, the quadratic curves (ellipses, parabolas and hyperbolas) do not have constant Frenet curvatures in general and have constant equi-affine curvatures. Below, we will exhibit a class of examples of curves with constant equi-affine and non-constant Frenet curvatures in certain Riemannian and Lorentzian 2-manifolds different from Euclidean and Minkowski planes.

\bigskip
\noindent
{\bf Examples.} 
On the set $M^2 = \{(x,y)\in\mathbb R^2\colon x>0\}$, consider the pseudo-Riemannian metric 
\begin{align*}
  g = x^{-3} (dx\otimes dx + \omega dy\otimes dy),
\end{align*}
where $\omega=\pm1$. The metric $g$ is Riemannian or Lorentzian according to whether $\omega=1$ or $\omega=-1$. Its non-zero Christoffel symbols are 
\begin{align*}
  {\Gamma_{11}}^1 = {\Gamma_{12}}^2 = {\Gamma_{21}}^2= - {\Gamma_{22}}^1 = - \frac{3 \omega}{2x}.
\end{align*}
The scalar curvature is equal to $-3x$, so the metric has non-constant Gauss curvature. The natural volume form is $\varOmega = 2 x^{-3} dx\wedge dy$. 

(i) In $(M^2,g)$, consider the curves $\alpha(t) = (\lambda, t), \ t\in\mathbb R$, where $\lambda$ is a positive constant. For such curves, 
\begin{align*}
  \alpha^{\,\prime}(t) = \frac{\partial}{\partial y}\bigg|_{\alpha(t)}, \quad
  \nu(t) = \frac{1}{\lambda^{3/2}}, \quad
  (\nabla_{\alpha^{\,\prime}}\alpha^{\,\prime})(t) 
     = \frac{3\omega}{2\lambda} \frac{\partial}{\partial x}\bigg|_{\alpha(t)}.
\end{align*}
Therefore, using (\ref{kap-riem}) and then (\ref{equfr}), we find that 
\begin{align*}
  \kappa_r(t) = - \frac{3\omega \lambda^{1/2}}{2},\quad 
  \kappa_a(t) = \frac{3^{4/3} \omega \lambda^{2/3}}{2^{4/3}}.
\end{align*} 

(ii) In $(M^2,g)$ with $\omega=1$ (the Riemannian case), consider the curves 
\begin{align*}
  \alpha(t) = (\lambda \cos t, y_0 + \lambda \sin t), 
  \ t\in\Big(-\frac{\pi}{2},\frac{\pi}{2}\Big), 
  \ y_0\in\mathbb R, \ \lambda\in\mathbb R, \ \lambda > 0. 
\end{align*}
Thus, we have 
\begin{align*}
  \alpha^{\,\prime}(t) &= \null- \lambda \sin t \frac{\partial}{\partial x}\bigg|_{\alpha(t)} 
                      + \lambda \cos t \frac{\partial}{\partial y}\bigg|_{\alpha(t)}, \\
  \nu(t) &= \frac{1}{\lambda^{1/2} (\cos t)^{3/2}}, \\
  (\nabla_{\alpha^{\,\prime}}\alpha^{\,\prime})(t) 
     &= \frac{\lambda}{2} (\cos t - 3 \sin t \tan t)
        \frac{\partial}{\partial x}\bigg|_{\alpha(t)} 
        + 2 \lambda \sin t \frac{\partial}{\partial y}\bigg|_{\alpha(t)}.
\end{align*}
Therefore, using (\ref{kap-riem}) and then (\ref{equfr}), we find that 
\begin{align*}
  \kappa_r(t) = - \left(\frac{\lambda}{4}\right)^{1/2} (\cos t)^{3/2},\quad
  \kappa_a(t) = - \left(\frac{\lambda}{4}\right)^{2/3}.
\end{align*}

(iii) In $(M^2,g)$ with $\omega=-1$ (the Lorentzian case), consider the curves 
\begin{align*}
  \alpha(t) = (\lambda \cosh t, y_0 + \lambda \sinh t), 
  \ y_0\in\mathbb R, \ \lambda\in\mathbb R, \lambda > 0, 
  \ t\in\mathbb R. 
\end{align*}
Thus, we have 
\begin{align*}
  \alpha^{\,\prime}(t) &= \lambda \sinh t \frac{\partial}{\partial x}\bigg|_{\alpha(t)} 
                      + \lambda \cosh t \frac{\partial}{\partial y}\bigg|_{\alpha(t)}, \\
  \nu(t) &= \frac{1}{\lambda^{1/2} (\cosh t)^{3/2}}, \\
  (\nabla_{\alpha^{\,\prime}}\alpha^{\,\prime})(t) 
     &= \null- \frac{\lambda}{2} (\cosh t + 3 \sinh t \tanh t) 
        \frac{\partial}{\partial x}\bigg|_{\alpha(t)} 
        - 2 \lambda \sinh t \frac{\partial}{\partial y}\bigg|_{\alpha(t)}.
\end{align*}
Therefore, using (\ref{kap-riem}) and then (\ref{equfr}), we find that
\begin{align*}
  \kappa_r(t) = \left(\frac{\lambda}{4}\right)^{1/2} (\cosh t)^{3/2},\quad
  \kappa_a(t) = \left(\frac{\lambda}{4}\right)^{2/3}.
\end{align*}


\end{document}